\documentclass[a4paper,twoside,11pt,reqno]{amsart}
\usepackage{fullpage}
\usepackage[T1]{fontenc}
\usepackage[utf8]{inputenc}

\usepackage{hyperref}
\usepackage[slantedGreeks, partialup, noDcommand]{kpfonts}
\usepackage{graphicx}
\usepackage[mathscr]{euscript}
\usepackage{caption}
\usepackage{tikz}
\usetikzlibrary{trees}
\usepackage{dutchcal}

\hypersetup{
     colorlinks   = true,
     citecolor    = black,
     linkcolor    = black,
     urlcolor     = black
}

\newcommand{\real}{\mathbb{R}}
\newcommand{\N}{A}
\newcommand{\D}{B}
\newcommand{\set}{\mathcal{S}}

\newcommand{\dd}{\textnormal{d}}
\newcommand{\expvar}{r}
\newcommand{\onevar}{w}

\newtheorem{theorem}{Theorem}
\newtheorem{lemma}[theorem]{Lemma}

\newtheorem{proposition}[theorem]{Proposition}
\newtheorem{corollary}[theorem]{Corollary}
\newtheorem{conjecture}[theorem]{Conjecture}

\newtheorem*{definition*}{Definition}

\title{Proof of the Holevo--Utkin conjecture on sharp $\ell_p$ norms for zero-sum vectors}

\author{Haonan Zhang}
\address{(H.~Z.)
Department of Mathematics,
University of South Carolina\\
Columbia, SC, 29208, USA.}
\email{haonanzhangmath@gmail.com}

\thanks{The author thanks Alexander Holevo and Andrey Utkin deeply for their careful reading and checking the proofs, as well as for their instructive suggestions that greatly improved the presentation of the paper. He is grateful to Paata Ivanisvili and Xinyuan Xie for valuable feedback on an earlier version of this paper. He also thanks ChatGPT (GPT-5 Pro) for helpful and stimulating discussions. The author is supported by NSF DMS-2453408. }

\begin{document}

\begin{abstract}
Let $d\ge 3$ and $p>0$. Let $\|x\|_p$ denote the $\ell_p$ (quasi-)norm of a $d$-dimensional vector $x$. Holevo and Utkin \cite{HU26} conjectured that for $0<p\le 1$,
\[
\min \left\{\frac{\|x\|_p}{\|x\|_2}:\vec{0}\neq x\in\mathbb R^d,\ \sum_{i=1}^d x_i=0\right\}
=2^{1/p-1/2};
\]
for $1<p<2$,
\[
\min \left\{\frac{\|x\|_p}{\|x\|_2}:\vec{0}\neq x\in\mathbb R^d,\ \sum_{i=1}^d x_i=0\right\}
=
\min\left\{2^{1/p-1/2},\left(\frac{(d-1)^{p/2}+(d-1)^{1-p/2}}{d^{p/2}}\right)^{1/p}\right\};
\]
and for $2<q<\infty$
\[
\max\left\{\frac{\|x\|_q}{\|x\|_2}:\vec{0}\neq x\in\mathbb R^d,\ \sum_{i=1}^d x_i=0\right\}
=
\max\left\{2^{1/q-1/2},\left(\frac{(d-1)^{q/2}+(d-1)^{1-q/2}}{d^{q/2}}\right)^{1/q}\right\}.
\]
They proved the $d=3$ case in \cite{HU26}. In this paper, we confirm the conjecture of the remaining cases $d\ge 4$.
\end{abstract}


\maketitle

\section{Introduction}
For any vector $x=(x_1,\dots, x_d)\in \real^d$, we use $\|x\|_p$ to denote its $\ell_p$ (quasi-)norm:
\begin{equation}
    \|x\|_p^p:=\sum_{1\le i\le d}|x_i|^p,\qquad p>0.
\end{equation}
It is an elementary estimate that for all $0<p\le 2\le q<\infty$
\begin{equation}
    \|x\|_q\le \|x\|_2\le \|x\|_p.
\end{equation}
Clearly, the constant $1$ is best possible and is attained by multiples of the coordinate vectors $e_i$. 

The main result of this paper is an improvement for vectors $x\in\real^d$ with zero sum, that is, $\sum_i x_i=0$. This was studied by Holevo and Utkin \cite{HU26}, motivated by the computation of accessible information for the ensemble of a ``quantum pyramid'', and we refer to 
\cite{HU25,HU26} and references therein for further discussion. They 
also suggest the connection to the Wehrl entropy problem for (the standard representation of) symmetric groups, and we defer the discussion to the end of the introduction. 

In \cite{HU26}, Holevo and Utkin conjectured that the optimal constants are given by the two possible families of optimizers
\[
(1,-1,0,\dots,0)
\]
and
\[
\left(d-1,-1,-1,\dots, -1\right),
\]
up to permutation and a global sign. More precisely, they made the following conjecture.

\begin{conjecture}\label{conj}
    Let $d\ge 3$ and $0<p<2<q<\infty$. Then
\begin{enumerate}
    \item for $0<p\le 1$,
\[
\min \left\{\frac{\|x\|_p}{\|x\|_2}:\vec{0}\neq x\in\mathbb R^d,\ \sum_{i=1}^d x_i=0\right\}
=2^{1/p-1/2};
\]
\item for $1<p<2$,
\[
\min \left\{\frac{\|x\|_p}{\|x\|_2}:\vec{0}\neq x\in\mathbb R^d,\ \sum_{i=1}^d x_i=0\right\}
=
\min\left\{2^{1/p-1/2},\left(\frac{(d-1)^{p/2}+(d-1)^{1-p/2}}{d^{p/2}}\right)^{1/p}\right\};
\]
    \item and for $2<q<\infty$
\[
\max\left\{\frac{\|x\|_q}{\|x\|_2}:\vec{0}\neq x\in\mathbb R^d,\ \sum_{i=1}^d x_i=0\right\}
=
\max\left\{2^{1/q-1/2},\left(\frac{(d-1)^{q/2}+(d-1)^{1-q/2}}{d^{q/2}}\right)^{1/q}\right\}.
\]
\end{enumerate}
\end{conjecture}


The $d=2$ case is trivial. Holevo and Utkin \cite{HU26} proved the case $d=3$ via a beautiful combination of trigonometric series expansion that works for $0<p<2$ and $2<q<4$, and a delicate one-dimensional analysis that handles $q>4$. Indeed, assume $x=(x_1,x_2,x_3)$ satisfies
\begin{equation}\label{d=3}
    \sum_{i=1}^{3}  x_i=0,\qquad \textnormal{and}\qquad \sum_{i=1}^{3}  x_i^2=1.
\end{equation}
The idea of trigonometric series expansion is based on the observation that there is an angle $\theta$ such that 
\[
x_j=\sqrt{\frac{2}{3}}\cos \left(\theta+\frac{2(j-1)\pi}{3}\right),\qquad j=1,2,3.
\]
This, together with the Euler formula and the sum of geometric progression, is the starting point of their proof for $0<p<2$ and $2<q<4$. However, this elegant approach does not seem to extend to $q>4$. To deal with the $q>4$ case, Holevo and Utkin \cite{HU26} rewrite the objective function 
$$F(x)=|x_1|^q+|x_2|^q+|x_3|^q$$
in terms of a new variable $u:=\frac{x_1-x_2}{x_1+x_2}\in [0,1]$ (assume $x_2\le x_1\le 0\le x_3$) under the constraint \eqref{d=3}
 \[g(u):=\left(\frac{2}{3+u^2}\right)^{\!q/2}
\left[
1+\frac{(1+u)^q+(1-u)^q}{2^q}
\right].\]
Then they reformulate the conjectured bound and reveal a monotonicity phenomenon arising in a careful analysis of the function $g$.

As for the case $d=3$ and $q=4$, a direct computation shows that, as observed in \cite{HU26}, for all $x=(x_1,x_2,x_3)$ satisfying
\[
\sum_{i=1}^{3}  x_i=0,\qquad \textnormal{and}\qquad \sum_{i=1}^{3}  x_i^2=1
\]
one always has $\sum_{i=1}^{3} x_i^4\equiv \frac12$. Thus, in this case, the problem is trivial, and all admissible vectors are optimizers. 


\medskip 

We will give another proof of the $d=3$ case in the last section; see Proposition \ref{prop:new} (1). The main result of this paper is the proof of the remaining cases. 

\begin{theorem}\label{thm:main}
    Conjecture \ref{conj} holds for all $d\ge 4$.
\end{theorem}

We remark that Holevo and Utkin \cite{HU26} already studied the general $d\ge 3$ via the Lagrange multiplier
\[
F(x,\lambda,\mu)=\sum_i |x_i|^\rho-\lambda\sum_i x_i-\mu\left(\sum_i x_i^2-1\right),\qquad \rho>0.
\]
They proved that the optimizers must be of the form 
\[
x=(\underbrace{a,\dots,a}_{d_0\text{ times}}, 
\underbrace{-b,\dots,-b}_{d_1\text{ times}},
\underbrace{-c,\dots,-c}_{d_2\text{ times}}),\qquad a\ge 0, \qquad b\ge c\ge 0.
\]
This reduces the conjecture to a simpler yet still involved optimization problem, and the conjecture is supported by numerical evidence in \cite{HU26}. 

The proof of our main theorem takes the above Lagrange multiplier as a starting point (for $\rho>1$). One key observation is that there are extra constraints on $d_j,j=0,1,2$, and we will see the essential challenge is to exclude the potential optimizer when $a>0,b>c>0$ and $(d_0,d_1,d_2)=(1,1,d-2)$. This further reduces the problem to a one-dimensional analysis similar to the case $d=3$ treated in \cite{HU26}. However, the analysis of this one-dimensional problem is still non-trivial, as in the case $d=3$. Our proof is inspired by the aforementioned change-of-variable trick in \cite{HU26}, but we have to use different techniques to handle different parameter regimes. 

\medskip

As remarked earlier, we shall not recall the full details of the motivation of Holevo and Utkin coming from the computation
of accessible information for the ensemble of a ``quantum pyramid'' \cite{HU25,HU26}. Here, we only formulate the result in terms of entropies. Recall that for any $\alpha\in (0,1)\cup (1,\infty)$, the $\alpha$-R\'enyi entropy $H_{\alpha}(p)$ of any probability density $p=(p_i)_{i=1}^{d}$ is defined as 
$$H_{\alpha}(p)=\frac{1}{1-\alpha}\log \sum_i p_i^\alpha.$$ 
In particular, it recovers the Shannon entropy $H(p):=-\sum_i p_i \log p_i$ when $\alpha\to 1$. 

Any $x\in \real^d$ with $\|x\|_2=1$ defines a probability density $P_x=(|x_i|^2)_{i=1}^{d}$. Then, combining the results of Holevo--Utkin \cite{HU26} and the above theorem, one has the following equivalent formulation, stated as a corollary. 

This entropy formulation is another way to express the main motivation from Holevo and Utkin; see \cite{HU25,HU26} and the references therein for further discussion. 

\begin{corollary}\label{cor:entropy}
Let $d\ge 3$. For $\alpha\in (0,1)\cup (1,\infty)$, we have 
\begin{equation}
    \min_{x} H_{\alpha}(P_x)=\min \left\{ \log 2, \, \frac{1}{1-\alpha}\log \left( \frac{(d-1)^{\alpha}+(d-1)^{1-\alpha}}{d^{\alpha}}\right)\right\},
\end{equation}
where $x$  runs over all vectors in $\real^d$ such that $\|x\|_2=1$ and $\sum_i x_i=0$.
Taking $\alpha\to 1$, one obtains 
\begin{equation}\label{eq:minimumoutput entropy}
    \min_{x} H(P_x)=\begin{cases}
        \log 2,& d\le 6;\\
        \log d-\frac{d-2}{d}\log (d-1),& d\ge 7,
    \end{cases}
\end{equation}
where $x$ runs over the same set of vectors as above.  
\end{corollary}

\subsection*{Complex vectors}
The main result extends to complex vectors by a convexity argument. In fact, assume $0<p<2$, and let $C_p$ be the minimum value for $\frac{\|x\|_p}{\|x\|_2}$, where $\vec{0}\neq x\in \real^d$ and $\sum_j x_j=0$. We claim that $C_p$ remains the optimal value when replacing real vectors $x$ with complex vectors $z$. Let $\vec{0}\neq z\in\mathbb{C}^d$ be such that $\|z\|_2^2=\sum_j |z_j|^2=1$ and $\sum_j z_j=0$. Writing $z=x+i y$ with $x,y\in \real^d$, one has $\|x\|_2^2+\|y\|_2^2=1$ and $\sum_j x_j=\sum_j y_j=0$. 

Assume $\vec{x},\vec{y}\neq \vec{0}$, since otherwise $\|z\|_p\ge C_p$ trivially. By concavity of $t\mapsto t^{p/2}$, 
\begin{align*}
    \|z\|_p^p=\sum_j (|x_j|^2+|y_j|^2)^{p/2}
&=\sum_j (\|x\|_2^2\cdot (|x_j|/\|x\|_2)^2+\|y\|_2^2\cdot (|y_j|/\|y\|_2)^2)^{p/2}\\
&\ge \sum_j \|x\|_2^2\cdot (|x_j|/\|x\|_2)^p+\|y\|_2^2\cdot (|y_j|/\|y\|_2)^p\\
&=\|x\|_2^2\cdot (\|x\|_p/\|x\|_2)^p+\|y\|_2^2\cdot (\|y\|_p/\|y\|_2)^p\\
&\ge C_p^p(\|x\|_2^2+\|y\|_2^2)\\
&=C_p^p.
\end{align*}
This proves the claim. The case for $q>2$ follows from a similar argument by convexity instead. 

\medskip

\subsection*{Minimum output entropy}
One may also interpret this complex form of Corollary~\ref{cor:entropy} as the minimum output entropy of a quantum-to-classical channel. Let \(P_{V_{\mathbb C}}\) be the orthogonal projection from \(\mathbb C^d\) onto
\[
V_{\mathbb C}:=\left\{z\in \mathbb C^d:\sum_i z_i=0\right\},
\]
and set
\[
v_i=P_{V_{\mathbb C}} e_i=e_i-\frac1d\vec{1},\qquad 1\le i\le d.
\]
Then
\[
\sum_{i=1}^d |v_i\rangle\langle v_i|=\textnormal{id}_{V_{\mathbb C}}.
\]
Therefore, the rank-one operators
\(E_i:=|v_i\rangle\langle v_i|\) form a positive operator-valued measure (POVM), and hence define a quantum-to-classical channel, see for instance \cite{NC10},
\[
\mathcal{E} (\rho)=\bigl(\operatorname{Tr}(E_1\rho),\dots, \operatorname{Tr}(E_d\rho)\bigr).
\]

The minimum output Shannon entropy is attained on pure states by concavity of entropy. Also, for a pure state $z\in V_{\mathbb C}$ with $\|z\|_2=1$, one has
\[
\mathcal E(|z\rangle\langle z|)=P_z.
\]
Thus, \eqref{eq:minimumoutput entropy} can also be understood as the minimum output entropy of the above quantum-to-classical channel $\mathcal E$, since 
\begin{equation}
    \min_{\rho: \textnormal{ state on } V_{\mathbb C}}H(\mathcal E (\rho))
    =\min_{z\in V_{\mathbb C}: \|z\|_2=1} H(\mathcal E(|z\rangle\langle z|)) 
    =\min_{z\in V_{\mathbb C}: \|z\|_2=1} H(P_z)
    =\min_{x\in V: \|x\|_2=1} H(P_x).
\end{equation}

\medskip 

\subsection*{Log-Sobolev inequality on complete graphs}

We also record a comparison with a sharp 2-log-Sobolev inequality of Diaconis and Saloff-Coste.
In a special case, the entropy estimate 
\eqref{eq:minimumoutput entropy} can also be viewed as a refinement of the classical $2$-log-Sobolev inequality for the simple random walk on the complete graph. 
Let $[d]=\{1,\dots,d\}$, let $\pi$ be the uniform probability measure, and write
\[
\operatorname{Ent}_\pi(f):=\pi (f\log f)-\pi(f)\log \pi(f)
=\frac{1}{d}\sum_i f(i)\log f(i)-\left(\frac{1}{d}\sum_i f(i)\right)\log \left(\frac{1}{d}\sum_i f(i)\right).
\]
Consider the random walk on the complete graph with generator
\[
Lf(i):=\frac1{d-1}\sum_{j\ne i} f(j)-f(i).
\]
It generates the semigroup $T_t=e^{tL}$, called the Potts semigroup \cite{GP23}. Its Dirichlet form is
\[
\mathcal E(f,f)=-\mathbb E_\pi[(Lf)f]
=\frac1{d(d-1)}\sum_{1\le i<j\le d}(f(i)-f(j))^2.
\]
Diaconis and Saloff-Coste \cite[Corollary A.5]{DSC96} proved the sharp log-Sobolev inequality
\begin{equation}\label{lsi}
\operatorname{Ent}_\pi(f^2)
\le 
\frac{(d-1)\log(d-1)}{d-2}
\mathcal E(f,f),\qquad  f:[d]\to \real. 
\end{equation}
See also \cite[Eq. (26)]{HU25}. On the zero-sum subspace $V=\{f:[d]\to \real: \sum_i f(i)=0\}$, however, Corollary~\ref{cor:entropy} gives a sharper restricted inequality 

\[
\operatorname{Ent}_\pi(f^2)
\le \frac{d-1}{d}\delta_d\,\mathcal E(f,f),
\qquad f\in V,
\]
where the entropy deficit is
\[
\delta_d:=\max_{x\in V,\ \|x\|_{2}=1}\bigl(\log d-H(P_x)\bigr)
=\begin{cases}
\log(d/2),& d\le 6,\\[2mm]
\dfrac{d-2}{d}\log(d-1),& d\ge 7.
\end{cases}
\]
 This improves \eqref{lsi} after restriction to zero-sum functions because 
 \[
 \frac{d-1}{d}\delta_d
 \le \frac{(d-1)\log(d-1)}{d-2}.
 \]

Another motivation comes from the Wehrl entropy problem going back to Lieb's solution to Wehrl's original conjecture on Glauber coherent states \cite{lieb78} corresponding to the Heisenberg group. It is beyond the scope of this paper to recall the full details and historical results. See, for example \cite{frank} and references therein. We only remark the potential connection here, which was mentioned briefly in \cite{HU26}.

Let $S_d$ be the symmetric group on $d$ letters. Consider the standard (irreducible) representation $\pi$ of $S_d$ over $V=\{x\in \real^d:\sum_i x_i=0\}$ of dimension $d-1$. The action is simply permuting the coordinates: $\pi(\sigma)x=(x_{\sigma^{-1}(i)})_i$. Unlike the known work on some Lie groups, $S_d$ is a finite group without a Lie algebra structure, so there is no notion of highest weight vectors for $\pi$. 


As the main result shows, the two types of vectors $e_i-e_j, 1\le i<j\le d$ and 
$d e_i-\vec{1}, 1\le i\le d$ are extremal for our optimization problems. 


Let us fix one of these two families of extremal unit vectors, call it $v$. Now, take any $x\in V$ with $\|x\|_2=1$, and consider the optimization of the generalized Wehrl entropy functional
\[
\int_{\sigma \in S_d} f(|\langle\pi(\sigma)x, v\rangle|^2 ) \dd \sigma = \frac{1}{d!}\sum_{\sigma \in S_d} f(|\langle\pi(\sigma)x, v\rangle|^2 )
\]
over all $x\in V$ for certain $f$. Take $f(t)=t^\alpha$ as usual. One asks if the minimum (or maximum) is attained when $x$ lies in the orbit of $v$ for $\alpha\in (0,1)$ (or $\alpha\in (1,\infty)$). 

If $v$ comes from $d e_i-\vec{1}, 1\le i\le d$, say $v=c_d(d-1,-1,-1,\dots, -1)$ with $c_d=1/\sqrt{(d-1)d}$, then this is exactly the optimization problem considered in this work because
\[
 \frac{1}{d!}\sum_{\sigma \in S_d} |\langle\pi(\sigma)x, v\rangle|^{2\alpha}=
 \frac{(dc_d)^{2\alpha}}{d}\sum_i |x_i|^{2\alpha}.
\]
However, our main result shows that the extremal value need not arise from the chosen orbit; instead, it exhibits a phase transition between two orbits of extremal vectors.

\medskip

\subsection*{Organization}  Section~\ref{section:p<1} treats the simplest case $0<p\le 1$. The proofs for the cases $1<p<2$ and $q>2$ use the Lagrange multiplier, reducing the problem to a one-dimensional analysis by examining the structure of the optimizers. The general strategies are similar, and we split them into Sections~\ref{section:1<p<2} and~\ref{section:q>2} for clarity. Section~\ref{section:r-analysis} provides a complementary view of the one-dimensional analysis and, in particular, gives another proof of the $d=3$ case proved by Holevo and Utkin \cite{HU26}.

\section{Proof of the main theorem: $0<p\le 1$}
\label{section:p<1}
In this section, we prove the main theorem for $0<p\le 1$. 
\begin{proof}[Proof of Theorem~\ref{thm:main} (1)]
Let
\[
\Lambda_+ := \{i : x_i > 0\}, \qquad \Lambda_- := \{i : x_i < 0\}.
\]
Since $\sum_i x_i = 0$, we have
\[
S := \sum_{i \in \Lambda_+} x_i = -\sum_{i \in \Lambda_-} x_i>0.
\]

For $0<p\le1$, we have for all $a,b \ge 0$,
\[
(a+b)^p \le a^p + b^p,
\]
so 
\[
\sum_{i \in \Lambda_+} x_i^p \ge S^p,
\qquad
\sum_{i \in \Lambda_-} |x_i|^p \ge S^p.
\]
Therefore,
\[
\|x\|_p^p=\sum_{i \in \Lambda_+} x_i^p +\sum_{i \in \Lambda_-} |x_i|^p \ge 2S^p.
\]

On the other hand,
\[
\|x\|_2^2
= \sum_{i \in \Lambda_+} x_i^2 + \sum_{i \in \Lambda_-} x_i^2
\le S^2 + S^2 = 2S^2.
\]
All combined, we have
\[
\frac{\|x\|_p}{\|x\|_2}
\ge 2^{\frac{1}{p}-\frac{1}{2}}.
\]
It is easy to check that the equality can be achieved by
\[
x=(1,-1,0,\dots,0).
\]
\end{proof}

\section{Proof of the main theorem: $1<p<2$}
\label{section:1<p<2}

We start with the Lagrange multiplier initiated in \cite{HU26}.
By homogeneity, it is enough to minimize
\[
F(x):=\sum_{i=1}^d |x_i|^p
\]
subject to
\[
\sum_{i=1}^d x_i=0,
\qquad
\sum_{i=1}^d x_i^2=1.
\]
Set
\[
\set:=
\left\{
 x\in\mathbb R^d:
 \sum_{i=1}^d x_i=0,\
 \sum_{i=1}^d x_i^2=1
\right\}.
\]
The set $\set$ is compact, so $F$ attains its minimum on $\set$.

Define
\[
g(x):=\sum_{i=1}^d x_i,\qquad h(x):=\sum_{i=1}^d x_i^2.
\]
At a minimizer $x\in \set$, the gradients
\[
\nabla g=(1,\dots,1),\qquad \nabla h=2x
\]
are linearly independent: if $\nabla h=c\nabla g$, then $x$ would be constant, and together with $\sum_i x_i=0$ this would force $x=0$, contradicting $\|x\|_2=1$. Hence the Lagrange multiplier theorem gives $\lambda,\mu\in\mathbb R$ such that
\[
\nabla F(x)=\lambda \nabla g(x)+\mu \nabla h(x),
\]
or equivalently
\begin{equation}\label{eq:EL-p}
p|x_i|^{p-2}x_i=\lambda+2\mu x_i,\qquad 1\le i\le d.
\end{equation}

We first isolate the case when the minimizer is of the form $\frac1{\sqrt2}(1,-1,0,\dots,0)$.

\begin{lemma}\label{lem:p-lambda0}
Let $x\in\mathbb R^d$ minimize
\[
F(x)=\sum_{i=1}^d |x_i|^p
\]
subject to the constraint set $\set$ above. Let $\lambda,\mu\in\mathbb R$ be the Lagrange multipliers in \eqref{eq:EL-p}. Then $\mu>0$.

Moreover, if $\lambda=0$, then
\[
F(x)\ge 2^{1-p/2},
\]
with equality exactly for vectors of the form
\[
\frac1{\sqrt2}(1,-1,0,\dots,0)
\]
up to permutation and a global sign.
\end{lemma}

\begin{proof}
Multiplying \eqref{eq:EL-p} by $x_i$ and summing over $i$ gives
\[
pF(x)=\lambda\sum_i x_i+2\mu\sum_i x_i^2=2\mu.
\]
Hence $\mu>0$.

If $\lambda=0$, then \eqref{eq:EL-p} becomes
\[
p|x_i|^{p-2}x_i=2\mu x_i,\qquad 1\le i\le d.
\]
So every nonzero coordinate satisfies
\[
p|x_i|^{p-2}=2\mu,
\]
and thus all nonzero coordinates have the same absolute value. Therefore, up to permutation and a global sign, $x$ is of the form
\[
x=(c,\dots,c,-c,\dots,-c,0,\dots,0)
\]
for some $c>0$ and some integer $m\ge1$, with $m$ positive and $m$ negative entries. Since $\|x\|_2=1$, the constant
\[
c=\frac1{\sqrt{2m}}.
\]
Therefore
\[
F(x)=2m\,c^p=2m(2m)^{-p/2}=(2m)^{1-p/2}.
\]
Since $1-p/2>0$, the minimum is achieved when $m=1$. This proves the lemma.
\end{proof}

We now assume $\lambda\neq0$. The next lemma proves that, up to permutation and a global sign, there is exactly one positive coordinate.

\begin{lemma}\label{lem:p-one-positive}
Under the assumptions above, if $\lambda\neq0$, then after multiplying $x$ by $-1$ if necessary and permuting coordinates, one has
\[
x=(a,-y_1,\dots,-y_{d-1}),\qquad a>0,\ y_j>0.
\]
\end{lemma}

\begin{proof}
By \eqref{eq:EL-p}, all $x_i$'s must be nonzero, otherwise $\lambda=0$, a contradiction. Replacing $x$ by $-x$ if necessary, we may assume $\lambda<0$.

Consider the function
\[
\phi(t):=pt^{p-1}-2\mu t,\qquad t>0.
\]
By a direct computation,
\[
\phi'(t)=p(p-1)t^{p-2}-2\mu,
\]
so $\phi'$ vanishes at exactly one positive point. Hence $\phi$ is strictly increasing and then strictly decreasing on $(0,\infty)$. Moreover,
\[
\phi(0)=0>\lambda,\qquad \lim_{t\to\infty}\phi(t)=-\infty.
\]
Therefore $\phi(t)=\lambda$ has exactly one positive solution. By \eqref{eq:EL-p}, every positive coordinate of $x$ must satisfy $\phi(t)=\lambda$, so all positive coordinates are equal.

Suppose now that there are at least two positive coordinates, say
\[
x_1=x_2=a>0.
\]
Let
\[
h:=e_1-e_2=(1,-1,0,\dots,0)\in\mathbb R^d.
\]
Consider the curve
\[
\gamma(t):=\frac{x+th}{\sqrt{1+2t^2}},\qquad |t|<a,
\]
or equivalently $\gamma(t)=(\gamma_i(t))_i$ with
\[
\gamma_1(t)=\frac{a+t}{\sqrt{1+2t^2}},\qquad
\gamma_2(t)=\frac{a-t}{\sqrt{1+2t^2}},\qquad
\gamma_i(t)=\frac{x_i}{\sqrt{1+2t^2}},\quad i\ge3.
\]
By definition,
\[
\sum_{i=1}^d \gamma_i(t)=0,\qquad \|\gamma(t)\|_2=1,
\]
so $\gamma(t)$ sits in the constraint set $\set$.

Consider the function
\[
\Phi(t):=F(\gamma(t))=\sum_{i=1}^d |\gamma_i(t)|^p,\qquad |t|<a.
\]
Since $\gamma(0)=x$, the function $\Phi$ has a local minimum at $t=0$. So $\Phi''(0)\ge0$.

Now set
\[
S_0:=\sum_{i=3}^d |x_i|^p.
\]
Then
\[
\Phi(t)=\frac{(a+t)^p+(a-t)^p+S_0}{(1+2t^2)^{p/2}}
= A(t)B(t)
\]
with
\[
A(t):=(a+t)^p+(a-t)^p+S_0,\qquad B(t):=(1+2t^2)^{-p/2}.
\]
A direct computation gives
\[
A(0)=F(x),\qquad A'(0)=0,\qquad A''(0)=2p(p-1)a^{p-2},
\]
and
\[
B(0)=1,\qquad B'(0)=0,\qquad B''(0)=-2p.
\]
Hence
\[
\Phi''(0)=A''(0)B(0)+A(0)B''(0)
=2p\bigl((p-1)a^{p-2}-F(x)\bigr).
\]
On the other hand, multiplying \eqref{eq:EL-p} by $x_i$ and summing over $i$ gives
\[
pF(x)=2\mu.
\]
This, together with \eqref{eq:EL-p} applied to $x_1=a$, yields
\[
pa^{p-1}=\lambda+2\mu a=\lambda+pF(x)a.
\]
Thus
\[
a^{p-2}=\frac{\lambda}{pa}+F(x),
\]
and therefore
\[
(p-1)a^{p-2}-F(x)
=\frac{p-1}{p}\frac{\lambda}{a}+(p-2)F(x)<0
\]
because $\lambda<0$ and $p-2<0$. Hence $\Phi''(0)<0$, a contradiction. Therefore there is at most one positive coordinate, and we conclude the proof.
\end{proof}

The next lemma gives more structure of the negative coordinates.

\begin{lemma}\label{lem:p-structure}
Let $n:=d-1$. Under the assumptions above, write
\[
x=(a,-y),\qquad y=(y_1,\dots,y_n),\qquad a>0,\ y_j>0.
\]
Then, after permutation,
\[
(y_1,\dots,y_n)=(b,c,\dots,c)
\]
for some $b\ge c>0$.
\end{lemma}

\begin{proof}
By \eqref{eq:EL-p}, each $y_j$ satisfies
\[
-py_j^{p-1}=\lambda-2\mu y_j,
\]
that is, each $y_j$ satisfies
\[
\phi(y_j)=-\lambda>0,
\]
for the same $\phi(t)=pt^{p-1}-2\mu t$ as above. Recall that $\phi$ is strictly increasing and then strictly decreasing, so the equation $\phi(t)=-\lambda$ has at most two positive roots. Thus, the $y_j$'s take at most two distinct values. After permutation,
\[
(y_1,\dots,y_n)=(
\underbrace{b,\dots,b}_{m\text{ times}},
\underbrace{c,\dots,c}_{n-m\text{ times}}
),
\qquad b>c>0,\qquad 1\le m\le n.
\]

When $m=n$, all negative coordinates are equal, and the constraints give
\[
x=
\left(
\sqrt{\frac{d-1}{d}},
-\frac1{\sqrt{d(d-1)}},
\dots,
-\frac1{\sqrt{d(d-1)}}
\right).
\]

Now assume $1\le m<n$. It remains to prove $m=1$. Suppose to the contrary that $m\ge2$. Set
\[
\bar y:=\left(\frac an,\dots,\frac an\right),\qquad
\rho:=\|y-\bar y\|_2,\qquad
v:=\frac1{\sqrt2}(e_1-e_2)\in\mathbb R^n.
\]
Since $y$ is not constant, we have $\rho>0$. By definition and the assumption that $y_1=y_2=b$,
\[
\sum_{j=1}^n (y_j-\bar y_j)=0,\qquad
\sum_{j=1}^n v_j=0,\qquad
\langle y-\bar y,v\rangle=0.
\]

Consider the curve $\widetilde\gamma(\theta)=(a,-\gamma(\theta))$ with
\[
\gamma(\theta):=\bar y+(y-\bar y)\cos\theta+\rho v\sin\theta,
\]
or equivalently
\[
\gamma_1(\theta)=\frac an+\left(b-\frac an\right)\cos\theta+\frac{\rho}{\sqrt2}\sin\theta,
\]
\[
\gamma_2(\theta)=\frac an+\left(b-\frac an\right)\cos\theta-\frac{\rho}{\sqrt2}\sin\theta,
\]
\[
\gamma_j(\theta)=\frac an+\left(b-\frac an\right)\cos\theta,\qquad 3\le j\le m,
\]
\[
\gamma_j(\theta)=\frac an+\left(c-\frac an\right)\cos\theta,\qquad m+1\le j\le n.
\]
Clearly, $\widetilde\gamma(0)=x$. Also, by definition,
\[
\sum_{j=1}^n \gamma_j(\theta)=\sum_{j=1}^n \bar y_j=a.
\]
Moreover,
\[
\|\gamma(\theta)\|_2^2
=\|\bar y\|_2^2+\cos^2\theta\,\|y-\bar y\|_2^2+\sin^2\theta\,\rho^2\|v\|_2^2
=\|\bar y\|_2^2+\|y-\bar y\|_2^2
=\|y\|_2^2,
\]
so $\widetilde\gamma(\theta)=(a,-\gamma(\theta))$ lies in the constraint set $\set$.

Since each $\gamma_j(\theta)$ is continuous and $\gamma_j(0)\in\{b,c\}$, all coordinates remain positive for $|\theta|$ sufficiently small. Therefore the function
\[
\Phi(\theta):=F(\widetilde\gamma(\theta))
= a^p+\sum_{j=1}^n \gamma_j(\theta)^p
\]
has a local minimum at $\theta=0$, so $\Phi''(0)\ge0$.

A direct computation gives
\[
\gamma_1'(0)=\frac{\rho}{\sqrt2},\qquad
\gamma_2'(0)=-\frac{\rho}{\sqrt2},\qquad
\gamma_j'(0)=0\quad\text{for }j\ge3,
\]
and
\[
\gamma_j''(0)=-(y_j-\bar y_j),\qquad 1\le j\le n.
\]
Hence
\[
\Phi''(0)
= p(p-1)\sum_{j=1}^n y_j^{p-2}\gamma_j'(0)^2
+ p\sum_{j=1}^n y_j^{p-1}\gamma_j''(0)
\]
\[
= p(p-1)\rho^2 b^{p-2}
- p\left[
mb^{p-1}\left(b-\frac an\right)
+(n-m)c^{p-1}\left(c-\frac an\right)
\right].
\]
Recall that $a=mb+(n-m)c$, so
\[
b-\frac an=\frac{n-m}{n}(b-c),\qquad
c-\frac an=-\frac{m}{n}(b-c),
\]
and
\[
\rho^2
= m\left(b-\frac an\right)^2+(n-m)\left(c-\frac an\right)^2
=\frac{m(n-m)}{n}(b-c)^2.
\]
Then we may rewrite
\[
\Phi''(0)
= p\frac{m(n-m)}{n}(b-c)^2
\left(
(p-1)b^{p-2}
-\frac{b^{p-1}-c^{p-1}}{b-c}
\right).
\]
Since $b>c$ and the function $s\mapsto s^{p-1}$ is strictly concave on $(0,\infty)$ (we used $1<p<2$ here), we have
\[
\frac{b^{p-1}-c^{p-1}}{b-c}>(p-1)b^{p-2}.
\]
Therefore $\Phi''(0)<0$, a contradiction. Thus $m=1$ and we finish the proof of the lemma.
\end{proof}

Now set $\ell:=d-2=n-1$. We know that a minimizer
\[
x=(a,-b,-c,\dots,-c)
\]
in the above lemma satisfies
\[
a=b+\ell c,\qquad a^2+b^2+\ell c^2=1.
\]
So we may represent the objective function $F(x)$ as a single-variable function. Define
\[
\onevar:=\frac cb\in[0,1].
\]
A simple computation gives
\[
b=b(\onevar)=\frac1{\sqrt{2+2\ell \onevar+\ell(\ell+1)\onevar^2}},\qquad
c=c(\onevar)=\onevar\,b(\onevar),\qquad
a=a(\onevar)=(1+\ell \onevar)b(\onevar),
\]
and it remains to study the minimum of
\[
\Psi(\onevar):=a(\onevar)^p+b(\onevar)^p+\ell c(\onevar)^p
=
\frac{(1+\ell \onevar)^p+1+\ell \onevar^p}
{\bigl(2+2\ell \onevar+\ell(\ell+1)\onevar^2\bigr)^{p/2}},
\qquad 0\le \onevar\le1.
\]
Notice that $\onevar=0$ and $\onevar=1$ correspond to
\[
\frac1{\sqrt2}(1,-1,0,\dots,0)
\quad\text{and}\quad
\left(
\sqrt{\frac{d-1}{d}},
-\frac1{\sqrt{d(d-1)}},
\dots,
-\frac1{\sqrt{d(d-1)}}
\right),
\]
respectively. The endpoint values are the desired possible minima
\[
\Psi(0)=2^{1-p/2},\qquad
\Psi(1)=\frac{(d-1)^{p/2}+(d-1)^{1-p/2}}{d^{p/2}}.
\]

The next lemma studies the function $\Psi$.

\begin{lemma}\label{lem:p-onevar}
Let $\ell\ge 1$ and $1<p<2$. Define
\[
\Psi(\onevar)=\frac{(1+\ell \onevar)^p+1+\ell \onevar^p}{\bigl(2+2\ell \onevar+\ell(\ell+1)\onevar^2\bigr)^{p/2}},
\qquad 0\le \onevar\le 1.
\]
Then $\Psi$ has at most one critical point in $(0,1)$, and any such critical point is a strict local maximum. In particular,
\[
\Psi(\onevar)>\min\{\Psi(0),\Psi(1)\},\qquad 0<\onevar<1.
\]
\end{lemma}

\begin{proof}
Write
\[
p=2+\alpha,\qquad -1<\alpha<0.
\]
As before,
\[
\Psi(\onevar)=A(\onevar)B(\onevar)^{-p/2},
\]
with
\[
A(\onevar)=(1+\ell \onevar)^p+1+\ell \onevar^p,
\qquad
B(\onevar)=2+2\ell \onevar+\ell(\ell+1)\onevar^2.
\]
Then
\[
\Psi'(\onevar)=p\ell\,B(\onevar)^{-p/2-1}\Delta(\onevar),
\]
where
\[
\Delta(\onevar)
=
(1-\onevar)\bigl(s^{p-1}-s\bigr)+(s+1)\bigl(\onevar^{p-1}-\onevar\bigr),
\qquad s=1+\ell \onevar.
\]

Set
\[
z:=\frac{1}{\onevar}\in(1,\infty).
\]
A direct computation gives
\[
\Delta(\onevar)=\onevar^{\alpha+2}\,\Theta(z),
\]
where
\[
\Theta(z):=\ell+2z+(z-1)(z+\ell)^{\alpha+1}-(\ell+z+1)z^{\alpha+1}.
\]
Hence
\[
\operatorname{sgn}\Psi'(\onevar)=\operatorname{sgn}\Theta(z),\qquad z=\frac{1}{\onevar}.
\]

We now show that $\Theta$ is strictly convex on $(1,\infty)$.
Differentiating twice,
\[
\Theta''(z)
=
2(\alpha+1)\bigl((z+\ell)^{\alpha}-z^{\alpha}\bigr)
+\alpha(\alpha+1)\Bigl((z-1)(z+\ell)^{\alpha-1}-(z+\ell+1)z^{\alpha-1}\Bigr).
\]
Using
\[
(z+\ell)^{\alpha}-z^{\alpha}
=
\alpha\int_0^\ell (z+t)^{\alpha-1}\,\dd t,
\]
we rewrite
\[
\Theta''(z)
=
\alpha(\alpha+1)\Xi(z),
\]
where
\[
\Xi(z):=
2\int_0^\ell (z+t)^{\alpha-1}\,\dd t
+(z-1)(z+\ell)^{\alpha-1}
-(z+\ell+1)z^{\alpha-1}.
\]

Since $\alpha-1<-1$, the function $u\mapsto u^{\alpha-1}$ is strictly convex on $(0,\infty)$. Hence for $t\in[0,\ell]$,
\[
(z+t)^{\alpha-1}\le \frac{\ell-t}{\ell}z^{\alpha-1}+\frac{t}{\ell}(z+\ell)^{\alpha-1}.
\]
Integrating over $(0,\ell)$ gives
\[
2\int_0^\ell (z+t)^{\alpha-1}\,\dd t
\le
\ell z^{\alpha-1}+\ell(z+\ell)^{\alpha-1}.
\]
Therefore
\begin{align*}
\Xi(z)
&\le
\ell z^{\alpha-1}+\ell(z+\ell)^{\alpha-1}
+(z-1)(z+\ell)^{\alpha-1}
-(z+\ell+1)z^{\alpha-1} \\
&=
(z+\ell-1)(z+\ell)^{\alpha-1}-(z+1)z^{\alpha-1} \\
&=
\bigl((z+\ell)^\alpha-z^\alpha\bigr)-(z+\ell)^{\alpha-1}-z^{\alpha-1}
<0,
\end{align*}
where the last inequality holds because $\alpha<0$ implies $(z+\ell)^\alpha-z^\alpha<0$.

Since $\alpha(\alpha+1)<0$, we conclude
\[
\Theta''(z)=\alpha(\alpha+1)\Xi(z)>0,\qquad z>1.
\]
Thus $\Theta$ is strictly convex on $(1,\infty)$.

Now $\Theta(1)=0$, and since $\alpha+1\in(0,1)$ we also have
\[
\Theta(z)=2z+O(z^{\alpha+1})\to\infty\qquad (z\to\infty).
\]
Recall that $\Theta(1)=0$, so $\Theta$ has at most one zero $z_0$ in $(1,\infty)$.

If no such $z_0$ exists, then $\Theta(z)>0$ for all $z>1$, hence $\Psi'(\onevar)>0$ on $(0,1)$ and $\Psi$ is strictly increasing on $(0,1)$.

If such $z_0$ exists, then strict convexity together with $\Theta(1)=\Theta(z_0)=0$ implies
\[
\Theta(z)<0\qquad \text{for }1<z<z_0,
\]
while $\Theta(z)>0$ for $z>z_0$ because $\Theta(z)\to\infty$ as $z\to\infty$. Therefore
\[
\Psi'(\onevar)>0\qquad \text{for }0<\onevar<\frac1{z_0},
\]
and
\[
\Psi'(\onevar)<0\qquad \text{for }\frac1{z_0}<\onevar<1.
\]
So $\onevar_0=1/z_0$ is the unique critical point in $(0,1)$, and it is a strict local maximum.

In either case,
\[
\Psi(\onevar)>\min\{\Psi(0),\Psi(1)\},\qquad 0<\onevar<1.
\]
This proves the lemma.
\end{proof}

Now we are ready to prove Theorem~\ref{thm:main} for  $1<p<2$.

\begin{proof}[Proof of Theorem~\ref{thm:main} (2)]
Let $x=(x_i)_i$ be a minimizer of
\[
F(x)=\sum_{i=1}^d |x_i|^p
\]
on
\[
\set=\left\{x\in\mathbb R^d:\sum_{i=1}^d x_i=0,\ \sum_{i=1}^d x_i^2=1\right\}.
\]
By the Lagrange multiplier theorem, there exist $\lambda,\mu\in\mathbb R$ such that
\[
p|x_i|^{p-2}x_i=\lambda+2\mu x_i,\qquad i=1,\dots,d.
\]

If $\lambda=0$, then Lemma~\ref{lem:p-lambda0} gives
\[
F(x)\ge 2^{1-p/2},
\]
with equality exactly for
\[
x=\frac1{\sqrt2}(1,-1,0,\dots,0)
\]
up to permutation and multiplication by $-1$.

Now assume $\lambda\neq0$. By Lemma~\ref{lem:p-one-positive}, after multiplying $x$ by $-1$ if necessary and permuting coordinates, we may write
\[
x=(a,-y_1,\dots,-y_{d-1}),\qquad a>0,\quad y_j>0.
\]
If all the $y_j$'s are equal, then necessarily
\[
x=\left(\sqrt{\frac{d-1}{d}},-\frac1{\sqrt{d(d-1)}},\dots,-\frac1{\sqrt{d(d-1)}}\right),
\]
and therefore
\[
F(x)=\Psi(1)=\frac{(d-1)^{p/2}+(d-1)^{1-p/2}}{d^{p/2}}.
\]

Otherwise, Lemma~\ref{lem:p-structure} shows that, up to permutation,
\[
x=(a,-b,-c,\dots,-c),\qquad b>c>0.
\]
With the parameter
\[
\onevar=\frac cb\in(0,1),
\]
the relations
\[
a=b+\ell c,\qquad a^2+b^2+\ell c^2=1,\qquad \ell=d-2,
\]
show that $F(x)=\Psi(\onevar),$ with $\Psi$ defined above. By Lemma~\ref{lem:p-onevar},
\[
\Psi(\onevar)>\min\{\Psi(0),\Psi(1)\},\qquad 0<\onevar<1.
\]
Hence such a vector cannot be a minimizer.

Therefore
\[
\min_{x\in \set}F(x)
=
\min\left\{
2^{1-p/2},
\frac{(d-1)^{p/2}+(d-1)^{1-p/2}}{d^{p/2}}
\right\}.
\]
Since $F(x)=\|x\|_p^p$ on $\set$, taking the $p$-th root proves Theorem~\ref{thm:main}(2).
\end{proof}

\section{Proof of the main theorem: $q>2$}
\label{section:q>2}
Again, we start with the Lagrange multiplier argument, and the proof strategy is similar to the $1<p<2$ case. We repeat the full details for the reader's convenience. However, the estimate for $\Psi$ will be different. By homogeneity, it is enough to maximize
\[
F(x):=\sum_{i=1}^d |x_i|^q
\]
subject to
\[
\sum_{i=1}^d x_i=0,
\qquad
\sum_{i=1}^d x_i^2=1.
\]
Set
\begin{equation}\label{constraint}
    \set:=\left\{x\in\mathbb R^d:\sum_{i=1}^d x_i=0,\ \sum_{i=1}^d x_i^2=1\right\}.
\end{equation}
The set $\set$ is compact, so $F$ attains its maximum on $\set$.

Define
\[
g(x):=\sum_{i=1}^d x_i,
\qquad
h(x):=\sum_{i=1}^d x_i^2.
\]
At a maximizer $x\in\set$, the gradients
\[
\nabla g=(1,\dots,1),
\qquad
\nabla h=2x
\]
are linearly independent: if $\nabla h=c\nabla g$, then $x$ would be constant, and together with $\sum_i x_i=0$ this would force $x=0$, contradicting $\|x\|_2=1$. Hence the Lagrange multiplier theorem gives $\lambda,\mu\in\mathbb R$ such that
\[
\nabla F(x)=\lambda\nabla g(x)+\mu\nabla h(x)
\]
or equivalently
\begin{equation}\label{eq:EL-main}
q|x_i|^{q-2}x_i=\lambda+2\mu x_i,
\qquad 1\le i\le d.
\end{equation}

We will collect some lemmas to analyze the structure of the optimizers before the full proof. The first lemma singles out the case when $2^{-1/2}(1,-1,0,\dots, 0)$ is an optimizer.

\begin{lemma}\label{lem:lambda0}
Let $x\in\mathbb R^d$ maximize
\[
F(x)=\sum_{i=1}^d |x_i|^q
\]
subject to the constraint set $\set$ above \eqref{constraint}.
Let $\lambda,\mu\in\mathbb R$ be the Lagrange multipliers in \eqref{eq:EL-main}.
Then $\mu>0$. Moreover, if $\lambda=0$, then
\[
F(x)\le 2^{1-q/2},
\]
with equality exactly for vectors of the form
\[
\frac1{\sqrt2}(1,-1,0,\dots,0)
\]
up to permutation and multiplication by $-1$.
\end{lemma}

\begin{proof}
Multiplying by $x_i$ on both sides of \eqref{eq:EL-main} and summing over $i$ gives
\[
qF(x)=\lambda \sum_i x_i + 2\mu\sum_i x_i^2=2\mu.
\]
So $\mu>0$.
If $\lambda=0$, then \eqref{eq:EL-main} becomes
\[
q|x_i|^{q-2}x_i=2\mu x_i,\qquad 1\le i\le d.
\]
So every nonzero coordinate satisfies
\[
q|x_i|^{q-2}=2\mu
\]
and thus all nonzero coordinates have the same absolute value. Therefore, up to permutation and a global sign, $x$ is of the form
\[
x=(c,\dots,c,-c,\dots,-c,0,\dots,0)
\]
for some $c>0$ and some integer $m\ge1$, with $m$ positive and $m$ negative entries. Since $\|x\|_2=1$, the constant is $c=\frac1{\sqrt{2m}}$.
Therefore
\[
F(x)=2mc^q=2m(2m)^{-q/2}=(2m)^{1-q/2}.
\]
Since $q>2$, $1-q/2<0$. Hence, $F(x)=(2m)^{1-q/2}$ achieves the maximum when $m=1$.
\end{proof}

The next lemma proves that when $\lambda \neq 0$, up to permutation and a global sign, there is exactly one positive coordinate.

\begin{lemma}\label{lem:one-positive}
Under the assumptions of Lemma~\ref{lem:lambda0}, if $\lambda\neq0$, then after multiplying $x$ by $-1$ if necessary and permuting coordinates, one has
\[
x=(a,-y_1,\dots,-y_{d-1}),
\qquad a>0,
\qquad y_j>0.
\]
\end{lemma}

\begin{proof}
Replacing $x$ by $-x$ if necessary, we may assume $\lambda>0$. By \eqref{eq:EL-main}, all $x_i$'s must be nonzero, otherwise $\lambda=0$ leading to a contradiction.

Consider the function
\[
\varphi(t):=qt^{q-1}-2\mu t,
\qquad t>0.
\]
By a direct computation,
\[
\varphi'(t)=q(q-1)t^{q-2}-2\mu,
\]
so $\varphi'$ vanishes at exactly one positive point. Hence $\varphi$ is strictly decreasing and then strictly increasing on $(0,\infty)$. Moreover,
\[
\varphi(0)=0<\lambda,
\qquad
\lim_{t\to\infty }\varphi(t)=\infty.
\]
Therefore $\varphi(t)=\lambda$
has exactly one positive solution. By \eqref{eq:EL-main}, every positive coordinate of $x$ must satisfy $\varphi(t)=\lambda$, so all positive coordinates are equal.

Suppose now that there are at least two positive coordinates, say,
\[
x_1=x_2=a>0.
\]
Let
\[
h:=e_1-e_2=(1,-1,0,\dots,0)\in\real^d.
\]

Consider the curve
\[
\gamma(t):=\frac{x+th}{\sqrt{1+2t^2}},
\qquad |t|<a,
\]
or $\gamma(t)=(\gamma_i(t))_i$ with
\[
\gamma_1(t)=\frac{a+t}{\sqrt{1+2t^2}},
\qquad
\gamma_2(t)=\frac{a-t}{\sqrt{1+2t^2}},
\qquad
\gamma_i(t)=\frac{x_i}{\sqrt{1+2t^2}},
\quad i\ge3.
\]
By definition,
\[
\sum_{i=1}^d \gamma_i(t)=0,
\qquad
\|\gamma(t)\|_2=1,
\]
so $\gamma(t)$ remains in the constraint set $\set$.

A direct computation gives
\[
\gamma_i'(t)=\frac{h_i-2t x_i}{(1+2t^2)^{3/2}},
\qquad
\gamma_i''(t)=\frac{-2x_i-6t h_i+8t^2x_i}{(1+2t^2)^{5/2}},
\]
with $h=(h_i)_i$ and $h_1=1,h_2=-1, h_i=0, i\ge 3$.
Consider the function
\[
\Phi(t):=F(\gamma(t))=\sum_{i=1}^d |\gamma_i(t)|^q=\sum_{i=1}^d \gamma_i(t)^q,\qquad |t|<a.
\]
Since $\gamma (0)=x$, $\Phi$ has a local maximum at $t=0$. So $\Phi''(0)\le0.$

Now we compute $\Phi''(0)$ directly. Setting
\[
S:=\sum_{i=3}^d |x_i|^q,
\]
we have
\[
\Phi(t)=\frac{(a+t)^q+(a-t)^q+S}{(1+2t^2)^{q/2}}=\N(t)\D(t)
\]
with
\[
\N(t):=(a+t)^q+(a-t)^q+S,
\qquad
\D(t):=(1+2t^2)^{-q/2}.
\]
A direct computation gives
\[
\N(0)=F(x),\qquad \N'(0)=0,\qquad \N''(0)=2q(q-1)a^{q-2},
\]
and
\[
\D(0)=1,\qquad \D'(0)=0,\qquad \D''(0)=-2q.
\]
Hence
\begin{equation}\label{eq:Theta-second-main}
\Phi''(0)=\N''(0)\D(0)+\N(0)\D''(0)=2q\bigl((q-1)a^{q-2}-F(x)\bigr).
\end{equation}
On the other hand, multiplying \eqref{eq:EL-main} by $x_i$ and summing over $i$ gives
\[
qF(x)=\lambda \sum_i x_i +2\mu \sum_i x_i^2=2\mu.
\]
This, together with \eqref{eq:EL-main} applied to $x_1=a$, yields
\[
qa^{q-1}=\lambda+2\mu a=\lambda+qF(x)a.
\]
Thus
\[
a^{q-2}=\frac{\lambda}{qa}+F(x),
\]
and therefore
\[
(q-1)a^{q-2}-F(x)=\frac{q-1}{q}\frac{\lambda}{a}+(q-2)F(x)>0.
\]
By \eqref{eq:Theta-second-main}, this implies $\Phi''(0)>0$, a contradiction.
Therefore, there is at most one positive coordinate, and we conclude the proof.
\end{proof}

The next lemma gives more structure of the negative coordinates.

\begin{lemma}\label{lem:structure-reduction}
Let $n:=d-1$. Under the assumptions of Lemma~\ref{lem:one-positive}, write
\[
x=(a,-y),\qquad y=(y_1,\dots,y_n),
\qquad a>0,
\qquad y_j>0.
\]
Then, after permutation,
\[
(y_1,\dots,y_n)=(b,c,\dots,c)
\]
for some $b\ge c>0$.
\end{lemma}

\begin{proof}
By \eqref{eq:EL-main}, each $y_j$ satisfies
\[
qy_j^{q-1}=-\lambda+2\mu y_j,
\]
that is, each $y_j$ satisfies
\[
\varphi(y_j)=-\lambda<0,
\]
for the same $\varphi(t)=qt^{q-1}-2\mu t$ as above.
Recall that $\varphi(0)=0$, $\varphi$ is strictly decreasing and then strictly increasing, the equation $\varphi(t)=-\lambda$
has at most two positive roots. Thus, the $y_j$'s take at most two distinct values. After permutation,
\[
(y_1,\dots,y_n)=(\underbrace{b,\dots,b}_{m\text{ times}},\underbrace{c,\dots,c}_{n-m\text{ times}}),
\qquad b>c>0,
\qquad 1\le m\le n.
\]
When $m=n$, all negative coordinates are equal, and the constraints give the endpoint vector
\[
x=\left(\sqrt{\frac{d-1}{d}},-\frac1{\sqrt{d(d-1)}},\dots,-\frac1{\sqrt{d(d-1)}}\right).
\]

Now we assume $1\le m<n$, i.e., there are two distinct values of negative coordinates. It remains to prove $m=1$. Suppose to the contrary that $m\ge2$. Set
\[
\bar y:=\left(\frac{a}{n},\dots,\frac{a}{n}\right),
\qquad
\rho:=\|y-\bar y\|_2,
\qquad
v:=\frac1{\sqrt2}(e_1-e_2)\in \real^n.
\]
Since $y$ is not constant, we have $\rho>0$. By definition and the assumption that $y_1=y_2=b$,
\[
\sum_{j=1}^n (y_j-\bar y_j)=0,
\qquad
\sum_{j=1}^n v_j=0,
\qquad
\langle y-\bar y,v\rangle=0.
\]

Consider the curve $\widetilde{\gamma}(\theta)=(a,-\gamma(\theta))$ with
\[
\gamma(\theta):=\bar y+(y-\bar y)\cos \theta+\rho v\sin \theta,
\]
or equivalently
\[
\gamma_1(\theta)=\frac{a}{n}+\left(b-\frac{a}{n}\right)\cos \theta+\frac{\rho}{\sqrt2}\sin \theta,
\]
\[
\gamma_2(\theta)=\frac{a}{n}+\left(b-\frac{a}{n}\right)\cos \theta-\frac{\rho}{\sqrt2}\sin \theta,
\]
\[
\gamma_j(\theta)=\frac{a}{n}+\left(b-\frac{a}{n}\right)\cos \theta,
\qquad 3\le j\le m,
\]
\[
\gamma_j(\theta)=\frac{a}{n}+\left(c-\frac{a}{n}\right)\cos \theta,
\qquad m+1\le j\le n.
\]
Clearly, $\widetilde{\gamma}(0)=x$. Also, by definition,
\[
\sum_{j=1}^n \gamma_j(\theta)=\sum_{j=1}^n \bar y_j=a.
\]
Moreover,
\[
\|\gamma(\theta)\|_2^2
=\|\bar y\|_2^2+\cos^2(\theta)\|y-\bar y \|_2^2+\sin^2(\theta)\rho^2\|v\|_2^2
=\|\bar y\|_2^2+\|y-\bar y \|_2^2
=\|y\|_2^2,
\]
so $\widetilde{\gamma}(\theta)=(a,-\gamma(\theta))$ lies in the constraint set $\set$.

Since each $\gamma_j(\theta)$ is continuous and $\gamma_j(0)\in\{b,c\}$, all coordinates remain positive for $|\theta|$ sufficiently small. Therefore the function
\[
\Phi(\theta):=F(\widetilde{\gamma}(\theta))=a^q+\sum_{j=1}^n \gamma_j(\theta)^q
\]
has a local maximum at $\theta=0$, so $\Phi''(0)\le0$.

A direct computation gives
\[
\gamma_1'(0)=\frac{\rho}{\sqrt2},
\qquad
\gamma_2'(0)=-\frac{\rho}{\sqrt2},
\qquad
\gamma_j'(0)=0 \textnormal{ for } j\ge3,
\]
and
\[
\gamma_j''(0)=-(y_j-\bar y_j),
\qquad 1\le j\le n.
\]
Hence
\begin{align*}
    \Phi''(0)
    &=q(q-1)\sum_{j=1}^n y_j^{q-2}\gamma_j'(0)^2+q\sum_{j=1}^n y_j^{q-1}\gamma_j''(0)\\
&=q(q-1)\rho^2 b^{q-2}-q\Bigl[m b^{q-1}\Bigl(b-\frac{a}{n}\Bigr)+(n-m)c^{q-1}\Bigl(c-\frac{a}{n}\Bigr)\Bigr].
\end{align*}

Recall that $a=mb+(n-m)c$, so
\[
b-\frac{a}{n}=\frac{n-m}{n}(b-c),
\qquad
c-\frac{a}{n}=-\frac{m}{n}(b-c),
\]
and
\[
\rho^2=m\left(b-\frac{a}{n}\right)^2+(n-m)\left(c-\frac{a}{n}\right)^2
=\frac{m(n-m)}{n}(b-c)^2.
\]
Then we may rewrite
\[
\Phi''(0)=q\frac{m(n-m)}{n}(b-c)^2\left((q-1)b^{q-2}-\frac{b^{q-1}-c^{q-1}}{b-c}\right).
\]
Since $b>c$ and the function $s\mapsto s^{q-1}$ is strictly convex (we used $q>2$ here) on $(0,\infty)$, we have
\[
(q-1)b^{q-2}-\frac{b^{q-1}-c^{q-1}}{b-c}>0.
\]
Therefore $\Phi''(0)>0$, a contradiction. Thus $m=1$ and we finish the proof of the lemma.
\end{proof}

Now set $\ell:=d-2=n-1$. We know that the optimizer $x=(a,-b,-c,\dots, -c)$ in the above lemma satisfies
\[
a=b+\ell c, \qquad a^2+b^2+\ell c^2=1.
\]
So we may represent the objective function $F(x)$ as a single-variable function. Define
\[
\onevar:=\frac{c}{b}\in[0,1].
\]
A simple computation gives
\[
b=b(\onevar)=\frac{1}{\sqrt{2+2\ell \onevar+\ell(\ell+1)\onevar^2}},
\qquad
c=c(\onevar)=\onevar\,b(\onevar),
\qquad
a=a(\onevar)=(1+\ell \onevar)b(\onevar)
\]
and it remains to study the maximum of
\[
\Psi(\onevar):=a(\onevar)^q+b(\onevar)^q+\ell c(\onevar)^q
=
\frac{(1+\ell \onevar)^q+1+\ell \onevar^q}{\bigl(2+2\ell \onevar+\ell(\ell+1)\onevar^2\bigr)^{q/2}},
\qquad 0\le \onevar\le1.
\]
Notice that $\onevar=0$ and $\onevar=1$ correspond to
\[
\frac1{\sqrt2}(1,-1,0,\dots,0)
\]
and
\[
\left(\sqrt{\frac{d-1}{d}},-\frac1{\sqrt{d(d-1)}},\dots,-\frac1{\sqrt{d(d-1)}}\right),
\]
respectively. The endpoint values are the desired possible maxima:
\[
\Psi(0)=2^{1-q/2},
\qquad
\Psi(1)=\frac{(\ell+1)^{q/2}+(\ell+1)^{1-q/2}}{(\ell+2)^{q/2}}.
\]
The next lemma studies the function $\Psi$.

\begin{lemma}
\label{lem:onevar}
Let $\ell\ge2$ and $q>2$. Consider
\[
\Psi(\onevar)=\frac{(1+\ell \onevar)^q+1+\ell \onevar^q}{\bigl(2+2\ell \onevar+\ell(\ell+1)\onevar^2\bigr)^{q/2}},
\qquad 0\le \onevar\le1
\]
defined above. Then we have the following.
\begin{enumerate}
\item If $q\ge3$, then $\Psi$ is strictly increasing on $[0,1]$.
\item If $2<q<3$, then
\[
\Psi(\onevar)< \max\{\Psi(0),\Psi(1)\},\qquad \onevar\in (0,1).
\]
\end{enumerate}
\end{lemma}

\begin{proof}
Write $\Psi(\onevar)=\N(\onevar)\D(\onevar)^{-q/2}$ with
\[
\N(\onevar):=(1+\ell \onevar)^q+1+\ell \onevar^q,
\qquad
\D(\onevar):=2+2\ell \onevar+\ell(\ell+1)\onevar^2.
\]
Then
\[
\Psi'(\onevar)=q\ell \D(\onevar)^{-q/2-1}\,\Delta(\onevar),
\]
where
\[
\Delta(\onevar):=\D(\onevar)\bigl((1+\ell \onevar)^{q-1}+\onevar^{q-1}\bigr)-\N(\onevar)\bigl(1+(\ell+1)\onevar\bigr).
\]
Set $s:=1+\ell \onevar$. Then
\[
\D(\onevar)=s^2+1+\ell \onevar^2,
\qquad
1+(\ell+1)\onevar=s+\onevar,
\]
and we may simplify $\Delta$ as
\[
\Delta(\onevar)=(1-\onevar)s^{q-1}+(s+1)\onevar^{q-1}-(s+\onevar)
=(1-\onevar)(s^{q-1}-s)-(s+1)(\onevar-\onevar^{q-1}).
\]

(1) If $q\ge3$, then $q-1\ge2$.
By the mean value theorem, there exist $\xi\in [1,s]$ and $\eta\in[\onevar,1]$ such that
\[
s^{q-1}-s=s(s^{q-2}-1)
=(q-2)s(s-1)\xi^{q-3}
=(q-2)\onevar\ell s \xi^{q-3},
\]
and
\[
\onevar-\onevar^{q-1}=\onevar(1-\onevar^{q-2})
=(q-2)\onevar(1-\onevar)\eta^{q-3}.
\]
Substituting into the formula for $\Delta$ yields
\[
\Delta(\onevar)=(q-2)\onevar(1-\onevar)\Bigl(\ell s\,\xi^{q-3}-(s+1)\eta^{q-3}\Bigr).
\]
Since $q-3\ge0$ and $0\le \eta\le 1\le \xi$, we have $0\le \eta^{q-3}\le 1\le \xi^{q-3}$. Therefore
\[
\Delta(\onevar)\ge (q-2)\onevar(1-\onevar)(\ell s-(s+1)).
\]
Recalling $s=1+\ell \onevar$ and $\ell\ge2$, we get
\[
\ell s-(s+1)=\ell-2+\ell(\ell-1)\onevar> 0
\]
for all $0<\onevar<1$. Hence for all $0<\onevar<1$, $\Delta(\onevar)> 0$, and thus $\Psi'(\onevar)> 0$. This proves (1).

(2) Assume now that $2<q<3$. Write
\[
q=2+\alpha,
\qquad 0<\alpha<1.
\]
Then the formula for $\Delta$ becomes
\[
\Delta(\onevar)=(1-\onevar)(1+\ell \onevar)^{1+\alpha}+(2+\ell \onevar)\onevar^{1+\alpha}-1-(\ell+1)\onevar.
\]
Write
\[
u_\alpha(z):=\frac{z^\alpha-1}{z-1}, \qquad z\ne 1.
\]
Rearranging gives
\[
\Delta(\onevar)=\onevar(1-\onevar)H(\onevar),
\]
where
\[
H(\onevar):=\frac{(1+\ell \onevar)^{1+\alpha}-(1+\ell \onevar)}{\onevar}-\frac{2+\ell \onevar}{1-\onevar}(1-\onevar^\alpha)
=\ell s\,u_\alpha(s)-(s+1)u_\alpha(\onevar), \qquad s=1+\ell \onevar.
\]

We show that $\Delta$ has at most one root in $(0,1)$. Since $\Delta$ and $H$ have the same roots in $(0,1)$, it suffices to prove that $H$ is strictly increasing on $(0,1).$ The function $\Delta$ itself need not be monotone on $(0,1)$, so we study $H$ instead.

Recall that for $\alpha\in (0,1)$, one has the integral representation for all $z>0$
\begin{equation}
    z^{\alpha}=c_\alpha \int_0^\infty t^\alpha\left(\frac{1}{t}-\frac{1}{t+z}\right)\,\dd t, \qquad c_\alpha=\frac{\sin(\pi\alpha)}{\pi}>0.
\end{equation}
This gives
\[
u_\alpha(z)=c_\alpha \int_0^\infty \frac{t^\alpha}{(1+t)(z+t)}\,\dd t,
\]
and
\[
u_\alpha'(z)=-c_\alpha \int_0^\infty \frac{t^\alpha}{(1+t)(z+t)^2}\,\dd t.
\]

Now we compute $H'$. Recall that $s=1+\ell \onevar$, so
\[
H'(\onevar)
= \ell^2 u_\alpha(s)+\ell^2 s u_\alpha'(s)-\ell u_\alpha(\onevar)-(s+1)u_\alpha'(\onevar).
\]

Substituting the above integral formulas yields
\[
H'(\onevar)
=
c_\alpha\int_0^\infty \frac{t^\alpha}{1+t}
\left[
\frac{\ell^2}{s+t}
-\frac{\ell^2 s}{(s+t)^2}
-\frac{\ell}{\onevar+t}
+\frac{s+1}{(\onevar+t)^2}
\right]\dd t.
\]
Since
\[
\frac{\ell^2}{s+t}-\frac{\ell^2 s}{(s+t)^2}
=
\frac{\ell^2 t}{(s+t)^2},
\]
and
\[
-\frac{\ell}{\onevar+t}+\frac{s+1}{(\onevar+t)^2}
=
\frac{-(\ell \onevar+\ell t)+(s+1)}{(\onevar+t)^2}
=
\frac{2-\ell t}{(\onevar+t)^2},
\]
we obtain
\[
H'(\onevar)
=
c_\alpha\int_0^\infty \frac{t^\alpha}{1+t}
\left[
\frac{\ell^2 t}{(s+t)^2}
+
\frac{2-\ell t}{(\onevar+t)^2}
\right]\dd t.
\]

One might expect the sum of the two terms in the bracket to be positive, but this can be false. Therefore, in the first integral term, we make the change of variable
\[
t=\ell \tau-1,\qquad \tau\in (1/\ell,\infty),
\]
so that $1+t=\ell\tau, \ s+t=\ell(\onevar+\tau)$ and thus
\[
\int_0^\infty \frac{t^\alpha}{1+t}\frac{\ell^2 t}{(s+t)^2}\,\dd t
=
\int_{1/\ell}^\infty \frac{(\ell\tau-1)^{\alpha+1}}{\tau(\onevar+\tau)^2}\,\dd \tau.
\]
Renaming $\tau$ back to $t$, and splitting the second integral term of $H'$ accordingly yields
\begin{align*}
    H'(\onevar)
&=c_\alpha \int_{1/\ell}^\infty \frac{(\ell t-1)^{\alpha+1}}{t(\onevar+t)^2}\,\dd t+
c_\alpha\left(\int_0^{1/\ell}+\int_{1/\ell}^{\infty}\right)\frac{t^\alpha(2-\ell t)}{(1+t)(\onevar+t)^2}\,\dd t\\
&=
 c_\alpha\int_0^{1/\ell} \frac{t^\alpha(2-\ell t)}{(1+t)(\onevar+t)^2}\,\dd t+
c_\alpha\int_{1/\ell}^\infty \frac{G(t)}{(\onevar+t)^2}\,\dd t,
\end{align*}
where
\[
G(t):=\frac{(\ell t-1)^{\alpha+1}}{t}+\frac{t^\alpha(2-\ell t)}{1+t},
\qquad t> \frac1\ell.
\]

We claim that $G(t)>0$ for all $t> 1/\ell$. Indeed, if $1/\ell\le t\le 2/\ell$, then
$2-\ell t\ge 0$, so clearly $G(t)>0$. If $t>2/\ell$, then $\ell t-2>0$. The desired $G(t)>0$ is equivalent to
\[
\frac{1+t}{t}\cdot \frac{\ell t-1}{\ell t-2}\cdot \left(\frac{\ell t-1}{t}\right)^{\alpha}>1.
\]
The first two terms are strictly larger than one. For the last term, recall that $\ell \ge 2$ and $t\ge 2/\ell$:
\[
\frac{\ell t-1}{t}
=\ell -\frac{1}{t}\ge \ell-\frac{\ell}{2}=\frac{\ell}{2}\ge 1,
\]
yielding $\left(\frac{\ell t-1}{t}\right)^{\alpha}\ge 1$ because $\alpha>0$. This concludes the proof of $G(t)>0$ for $t>1/\ell$.

According to the above integral representation of $H'$, we actually proved that
$H'(\onevar)>0$, $\onevar\in(0,1)$. So,
$H$ is strictly increasing on $(0,1)$, and can have at most one zero $\onevar_0$. If such $\onevar_0\in (0,1)$ exists, then $H(\onevar)<0$ for $\onevar<\onevar_0$ and $H(\onevar)>0$ for $\onevar>\onevar_0$. Since $\Delta(\onevar)=\onevar(1-\onevar)H(\onevar)$, the same sign change holds for $\Delta$, and hence for $\Psi$ as well. So, $\onevar_0$ is an interior minimizer for $\Psi$. If no zero exists, then $H$, and thus $\Delta$ and $\Psi$, have no sign change on $(0,1)$. Thus $\Psi$ is either strictly increasing or decreasing on $(0,1)$.

Therefore, in either case, we always have
\[
\Psi(\onevar)<\max\{\Psi(0),\Psi(1)\},\qquad \onevar\in (0,1).
\]
This finishes the proof of the lemma. 
\end{proof}

Now we are ready to finish the proof of the main theorem for $q>2$.

\begin{proof}[Proof of Theorem~\ref{thm:main} (3)]
The case $d=3$ is due to Holevo and Utkin \cite{HU26}, so we may assume $d\ge4$.

Let $x=(x_i)_i$ be an optimizer. According to the above discussion, by the Lagrange multiplier theorem, there exist $\lambda,\mu\in\mathbb R$ such that
\[
q|x_i|^{q-2}x_i=\lambda+2\mu x_i,
\qquad i=1,\dots,d.
\]

If $\lambda=0$, Lemma~\ref{lem:lambda0} gives
\[
F(x)\le 2^{1-q/2}.
\]
and the equality is attained when
\[
x=\frac1{\sqrt2}(1,-1,0,\dots,0)
\]
up to permutation.

If $\lambda\ne0$, then replacing $x$ by $-x$ if necessary, we may assume $\lambda>0$. By Lemma~\ref{lem:one-positive}, after permutation,
\[
x=(a,-y_1,\dots,-y_{d-1}),
\qquad a>0,
\qquad y_j>0.
\]
In case all the $y_j$'s are equal, $x$ is of the form
\[
x=\left(\sqrt{\frac{d-1}{d}},-\frac1{\sqrt{d(d-1)}},\dots,-\frac1{\sqrt{d(d-1)}}\right).
\]
Otherwise, by Lemma~\ref{lem:structure-reduction}, up to permutation, $x$ has the form
\[
x=(a,-b,-c,\dots,-c),
\]
with $b>c>0$. With the parameter $\onevar=c/b\in (0,1)$, we can solve $a,b,c$ as functions of $\onevar$, and it reduces to maximizing
\[
\Psi(\onevar)=\frac{(1+\ell \onevar)^q+1+\ell \onevar^q}{\bigl(2+2\ell \onevar+\ell(\ell+1)\onevar^2\bigr)^{q/2}},
\qquad \ell=d-2,
\qquad 0<\onevar<1.
\]
Extend the definition of $\Psi$ to $[0,1]$ and its endpoint values are
\[
\Psi(0)=2^{1-q/2},
\qquad
\Psi(1)=\frac{(d-1)^{q/2}+(d-1)^{1-q/2}}{d^{q/2}}.
\]
If $q\ge3$, Lemma~\ref{lem:onevar} shows that $\Psi$ is strictly increasing on $(0,1)$, hence
\[
\Psi(\onevar)<\Psi(1),\qquad 0<\onevar<1.
\]
If $2<q<3$, Lemma~\ref{lem:onevar} gives
\[
\Psi(\onevar)<\max\{\Psi(0),\Psi(1)\},\qquad 0<\onevar<1.
\]
Therefore, $x$ cannot take the form
\[
x=(a,-b,-c,\dots,-c),\qquad a>0, \qquad b>c>0.
\]

All combined, we conclude that
\[
\max_{x\in \set}F(x)=\max\left\{2^{1-q/2},\frac{(d-1)^{q/2}+(d-1)^{1-q/2}}{d^{q/2}}\right\}.
\]
Since $F(x)=\|x\|_q^q$ on $\set$, this proves Theorem~\ref{thm:main}(3) by taking the $q$-th root.
\end{proof}

\section{A complementary analysis in the exponent}
\label{section:r-analysis}

We close with a complementary view of the one-dimensional problem. For a fixed integer $\ell\ge1$, recall the functions
\[
\Psi_\expvar(\onevar)=\frac{(1+\ell \onevar)^\expvar+1+\ell \onevar^\expvar}{\bigl(2+2\ell \onevar+\ell(\ell+1)\onevar^2\bigr)^{\expvar/2}},
\qquad 0\le \onevar\le1.
\]
A direct differentiation gives, for $0<\onevar<1$,
\[
\Psi_\expvar'(\onevar)=\expvar\ell\,\D(\onevar)^{-\expvar/2-1}\,\Delta_\onevar(\expvar),
\qquad
\D(\onevar)=2+2\ell \onevar+\ell(\ell+1)\onevar^2,
\]
where
\[
\Delta_\onevar(\expvar):=(1-\onevar)(1+\ell \onevar)^{\expvar-1}+(2+\ell \onevar)\onevar^{\expvar-1}-\bigl(1+(\ell+1)\onevar\bigr).
\]
In the proofs above, the exponent was fixed and the sign of this main factor in $\Psi_\expvar'(\onevar)$ was analyzed as a function of $\onevar$. We now fix $\onevar$ and instead regard $\Delta_\onevar(\expvar)$ as a function of the exponent $\expvar$. In particular, the following proposition (1) gives an alternative proof of the $d=3$ case obtained in \cite{HU26}.

\begin{proposition}\label{prop:new}
Let $\ell\ge1$, and let $\Psi_\expvar$ and $\Delta_\onevar$ be defined as above. For each $0<\onevar<1$, the function $\expvar\mapsto\Delta_\onevar(\expvar)$ is strictly convex on $(0,\infty)$ and satisfies
\[
\Delta_\onevar(1)=2(1-\onevar),
\qquad
\Delta_\onevar(2)=0.
\]
Consequently, if $0<\expvar<1$, then $\Psi_\expvar$ is strictly increasing on $[0,1]$. Moreover:
\begin{enumerate}
\item If $\ell=1$, then $\Psi_\expvar$ is strictly increasing on $[0,1]$ for $1<\expvar<2$ and $\expvar>4$, strictly decreasing on $[0,1]$ for $2<\expvar<4$, and
\[
\Psi_2(\onevar)\equiv 1,
\qquad
\Psi_4(\onevar)\equiv \frac12.
\]
In particular, for every $\expvar>0$, the extrema of $\Psi_\expvar$ on $[0,1]$ are attained at the endpoints.
\item If $\ell\ge2$ and $\expvar\ge3$, then $\Psi_\expvar$ is strictly increasing on $[0,1]$.
\end{enumerate}
\end{proposition}

\begin{proof}
Fix $0<\onevar<1$. A direct differentiation gives
\[
\Delta_\onevar''(\expvar)
=(1-\onevar)(1+\ell \onevar)^{\expvar-1}\log^2(1+\ell \onevar)
+(2+\ell \onevar)\onevar^{\expvar-1}\log^2\onevar>0,
\]
so $\expvar\mapsto\Delta_\onevar(\expvar)$ is strictly convex on $(0,\infty)$. Also,
\[
\Delta_\onevar(1)=(1-\onevar)+(2+\ell \onevar)-1-(\ell+1)\onevar=2(1-\onevar),
\qquad
\Delta_\onevar(2)=0.
\]
Since $\Delta_\onevar$ is strictly convex, for $0<\expvar<1$ we have
\[
\frac{\Delta_\onevar(1)-\Delta_\onevar(\expvar)}{1-\expvar}
<
\frac{\Delta_\onevar(2)-\Delta_\onevar(1)}{2-1}
=-\Delta_\onevar(1),
\]
and therefore
\[
\Delta_\onevar(\expvar)>(2-\expvar)\Delta_\onevar(1)=2(2-\expvar)(1-\onevar)>0.
\]
Hence $\Psi_\expvar'(\onevar)>0$ for every $0<\expvar<1$.

Assume first that $\ell=1$. Then a direct expansion gives
\[
\Delta_\onevar(4)=(1-\onevar)(1+\onevar)^3+(2+\onevar)\onevar^3-(1+2\onevar)=0.
\]
Since $\Delta_\onevar$ is strictly convex and vanishes at $2$ and $4$, we obtain
\[
\Delta_\onevar(\expvar)>0\quad\text{for }1<\expvar<2\text{ and }\expvar>4,
\qquad
\Delta_\onevar(\expvar)<0\quad\text{for }2<\expvar<4.
\]
Because the prefactor in $\Psi_\expvar'(\onevar)$ is positive, the same sign statements hold for $\Psi_\expvar'(\onevar)$. Hence $\Psi_\expvar$ is strictly increasing on $(0,1)$ when $1<\expvar<2$ or $\expvar>4$, and strictly decreasing on $(0,1)$ when $2<\expvar<4$.

When $\expvar=2$, we have
\[
\Psi_2(\onevar)=\frac{(1+\onevar)^2+1+\onevar^2}{2+2\onevar+2\onevar^2}\equiv 1.
\]
When $\expvar=4$, the identity $\Delta_\onevar(4)=0$ shows that $\Psi_4'(\onevar)=0$ for all $0<\onevar<1$, so $\Psi_4$ is constant on $[0,1]$. Evaluating at $\onevar=0$ gives
\[
\Psi_4(\onevar)\equiv \Psi_4(0)=\frac{2}{2^2}=\frac12.
\]
This proves (1).

Assume now that $\ell\ge2$. A direct expansion yields
\[
\Delta_\onevar(3)
=(1-\onevar)(1+\ell \onevar)^2+(2+\ell \onevar)\onevar^2-1-(\ell+1)\onevar
=\onevar(1-\onevar)\bigl(\ell-2+\ell(\ell-1)\onevar\bigr)>0.
\]
We claim that $\Delta_\onevar(\expvar)>0$ for every $\expvar\ge3$. Indeed, if there were some $\expvar_0>3$ with $\Delta_\onevar(\expvar_0)\le0$, then convexity on the interval $[2,\expvar_0]$ together with $\Delta_\onevar(2)=0$ would imply
\[
\Delta_\onevar(3)\le \frac{\expvar_0-3}{\expvar_0-2}\Delta_\onevar(2)+\frac{1}{\expvar_0-2}\Delta_\onevar(\expvar_0)\le0,
\]
a contradiction. Hence $\Delta_\onevar(\expvar)>0$ for all $\expvar\ge3$. Therefore $\Psi_\expvar'(\onevar)>0$ for all $0<\onevar<1$, which proves (2).
\end{proof}

\end{document}